\documentclass[12pt,twoside]{amsart}
\usepackage{amssymb,amscd,amsxtra,calc,mathrsfs}
\usepackage{amsmath}
\usepackage{xypic}

\title[Berman-Gibbs stability]{On Berman-Gibbs stability and K-stability 
of $\mathbb{Q}$-Fano varieties}
\author{Kento Fujita} 
\date{\today}
\address{Department of Mathematics, Faculty of Science, 
Kyoto University, Kyoto 606-8502, Japan}
\email{fujita@math.kyoto-u.ac.jp}

\newcommand{\pr}{\mathbb{P}}

\newcommand{\Z}{\mathbb{Z}}
\newcommand{\Q}{\mathbb{Q}}

\newcommand{\A}{\mathbb{A}}
\newcommand{\G}{\mathbb{G}}

\newcommand{\Exc}{\operatorname{Exc}}

\newcommand{\codim}{\operatorname{codim}}

\newcommand{\Det}{\operatorname{Det}}
\newcommand{\lct}{\operatorname{lct}}
\newcommand{\DF}{\operatorname{DF}}
\newcommand{\sgn}{\operatorname{sgn}}
\newcommand{\sI}{\mathcal{I}}

\newcommand{\sO}{\mathcal{O}}

\newcommand{\sL}{\mathcal{L}}

\newcommand{\sF}{\mathcal{F}}
\newcommand{\sB}{\mathcal{B}}
\newcommand{\sA}{\mathcal{A}}

\newcommand{\yI}{\mathscr{I}}
\newcommand{\da}{\mathfrak{a}}
\newcommand{\db}{\mathfrak{b}}
\newcommand{\dS}{\mathfrak{S}}

\newtheorem{thm}{Theorem}[section]
\newtheorem{lemma}[thm]{Lemma}
\newtheorem{proposition}[thm]{Proposition}

\newtheorem{claim}[thm]{Claim}

\theoremstyle{definition}
\newtheorem{definition}[thm]{Definition}
\newtheorem{remark}[thm]{Remark}

\newtheorem*{ack}{Acknowledgments}

\begin{document}

\maketitle 

\begin{abstract}
The notion of Berman-Gibbs stability was originally introduced by 
Robert Berman for $\mathbb{Q}$-Fano varieties $X$. 
We show that the pair $(X, -K_X)$ 
is K-stable (resp.\ K-semistable) provided that $X$ is Berman-Gibbs stable 
(resp.\ semistable). 
\end{abstract}

\tableofcontents

\section{Introduction}\label{intro_section}

One of the most important problem for the study of $\Q$-Fano varieties $X$ 
(i.e., projective log-terminal varieties with $-K_X$ ample $\Q$-Cartier)
is to determine whether the pairs $(X, -K_X)$ are K-stable or not (for the notion 
of K-stability, see Section \ref{K_section}). 
Recently, Robert Berman introduced 
a new stability of $X$, which he calls Gibbs stability, 
and its variants. The main purpose of this paper is to show that, 
slightly modifying the definition (we rename it as \emph{Berman-Gibbs stability}), 
it implies the K-stability in Donaldson's \cite{don} and Tian's \cite{tian1} sense. 
In particular, by \cite{CDS1, CDS2, CDS3, tian2}, 
it implies the existence of K\"ahler-Einstein metric 
if $X$ is smooth and the base field is the complex number field. 
We remark that Robert Berman showed in 
\cite[Theorem 7.3]{B} that strongly Gibbs stable Fano manifolds defined over 
the complex number field admit K\"ahler-Einstein metrics, where the notion of 
strong Gibbs stability is stronger than the notion of Berman-Gibbs stability. 
Now we define the notion of Berman-Gibbs stability. 
(We remark that the notion of Berman-Gibbs stability is slightly weaker than 
the notion of uniform Gibbs stability. For detail, see \cite[Section 7]{B}.)

\begin{definition}\label{det_dfn}
Let $X$ be a projective variety and $L$ be a globally generated Cartier divisor on $X$. 
Set $N:=h^0(X, \sO_X(L))$ and 
$\phi:=\phi_{|L|}\colon X\to \pr^{N-1}$, where $\phi_{|L|}$ is a morphism defined by 
the complete linear system $|L|$. 
Consider the morphism $\Phi\colon X^N\to (\pr^{N-1})^N$ defined by the copies 
of $\phi$, that is, $\Phi(x_1,\dots,x_N):=(\phi(x_1),\dots,\phi(x_N))$ 
for $x_1,\dots,x_N\in X$. 
Let $\Det_N\subset(\pr^{N-1})^N$ be the divisor defined by the 
equation $\det(x_{ij})_{1\leq i, j\leq N}=0$, where 
\[
(x_{11}:\dots:x_{1N};\cdots\cdots;x_{N1}:\dots:x_{NN})
\]
are the multi-homogeneous 
coordinates of $(\pr^{N-1})^N$. 
We set the divisor $D_{X, L}\subset X^N$ defined by $D_{X, L}:=\Phi^*\Det_N$. 
\end{definition}

\begin{remark}\label{det_rmk}
The divisor $D_{X, L}\subset X^N$ is defined uniquely by $X$ and the linear equivalence 
class of $L$. In particular, the definition is independent of the choice of the basis 
of $H^0(X, \sO_X(L))$. 
\end{remark}

\begin{definition}[{\cite[(7.2)]{B}}]\label{B_dfn}
Let $X$ be a $\Q$-Fano variety. 
For $k\in\Z_{>0}$ with $-kK_X$ Cartier and globally generated, we set 
$N:=N_k:=h^0(X, \sO_X(-kK_X))$ and $D_k:=D_{X, -kK_X}\subset X^N$. 
Set 
\[
\gamma(X):=\liminf_{
\substack{k\to\infty\\
-kK_X:\text{{ Cartier}}
}}
\left(\lct_{\Delta_X}\left(X^N, \frac{1}{k}D_k\right)\right),
\]
where $\Delta_X(\simeq X)$ is the diagonal, that is,  
\[
\Delta_X:=\{(x,\dots,x)\in X^N\,\,|\,\,x\in X\}\subset X^N,
\]
and $\lct_{\Delta_X}(X^N, (1/k)D_k)$ is the log-canonical threshold 
(see \cite[\S 9]{L}) of the pair 
$(X^N, (1/k)D_k)$ around $\Delta_X$, that is, 
\[
\lct_{\Delta_X}\left(X^N, \frac{1}{k}D_k\right):=\sup\left\{c\in\Q_{>0}\,\,\Big|
\left(X^N, \frac{c}{k}D_k\right): 
\begin{matrix}
\text{ log-canonical}\\
\text{around }\Delta_X
\end{matrix}
\right\}.
\]
We say that $X$ is \emph{Berman-Gibbs stable} 
(resp.\ \emph{Berman-Gibbs semistable}) if $\gamma(X)>1$ 
(resp.\ $\gamma(X)\geq 1$).
\end{definition}

We show in this paper that Berman-Gibbs stability implies K-stability for any 
$\Q$-Fano variety. More precisely, we show the following: 

\begin{thm}[Main Theorem]\label{mainthm}
Let $X$ be a $\Q$-Fano variety. 
If $X$ is Berman-Gibbs stable $($resp.\ Berman-Gibbs semistable$)$, then 
the pair $(X, -K_X)$ is K-stable $($resp.\ K-semistable$)$. 
\end{thm}

Now we explain how this article is organized. In Section \ref{K_section}, 
we recall the notion and basic properties of K-stability. 
In Section \ref{multi_section}, we recall the notion and basic properties of 
multiplier ideal sheaves, which is a powerful tool to determine how much 
the singularities of given divisors or given ideal sheaves are mild. 
In Section \ref{P1_section}, we determine whether the projective line $\pr^1$ is 
Berman-Gibbs stable or not. We will see that $\pr^1$ is Berman-Gibbs semistable 
but is not Berman-Gibbs stable. 
In Section \ref{key_section}, we prove the key propositions in order to prove 
Theorem \ref{mainthm}. We will prove in Proposition \ref{IJ_prop} that 
Berman-Gibbs stability of $X$ implies that the singularity of a given certain ideal sheaf 
on $X\times\A^1$ is somewhat mild. 
The strategy of the proof of Proposition \ref{IJ_prop} is 
to see their multiplier ideal sheaves in detail. 
In Section \ref{proof_section}, we prove Theorem \ref{mainthm}. 
By combining the results in \cite{OS}, 
Proposition \ref{IJ_prop}, and by some numerical arguments, 
we can prove Theorem \ref{mainthm}.

\begin{ack}
The author would like to thank Professor Robert Berman and 
Doctor Yuji Odaka for his helpful comments. 
Especially, Doctor Yuji Odaka informed him the interesting paper \cite{B} and 
Professor Robert Berman pointed out Remark \ref{Berman_rmk}. 
The author is partially supported by a JSPS Fellowship for Young Scientists. 
\end{ack}

Throughout this paper, we work in the category of algebraic (separated and of 
finite type) scheme over a fixed algebraically closed field $\Bbbk$ of characteristic 
zero. A \emph{variety} means a reduced and irreducible algebraic scheme. 
For the theory of minimal model program, we refer the readers to \cite{KoMo}; 
for the theory of multiplier ideal sheaves, we refer the readers to \cite{L}. 
For varieties $X_1,\dots,X_N$, let $p_j\colon \prod_{1\leq i\leq N}X_i\to X_j$ be 
the $j$-th projection morphism for any $1\leq j\leq N$.

\section{Preliminaries}\label{prelim_section}

In this section, we correct some definitions. 

\subsection{K-stability}\label{K_section}

We quickly recall the definition and basic properties of K-stability. 
For detail, for example, see \cite{odk} and references therein. 

\begin{definition}[{see \cite{tian1, don, RT, odk, LX}}]\label{K_dfn}
Let $X$ be a $\Q$-Fano variety of dimension $n$. 
\begin{enumerate}
\renewcommand{\theenumi}{\arabic{enumi}}
\renewcommand{\labelenumi}{(\theenumi)}
\item\label{K_dfn1}
A \emph{flag ideal} $\yI$ is an ideal sheaf $\yI\subset\sO_{X\times\A^1_t}$ 
of the form 
\[
\yI=I_M+I_{M-1}t+\cdots+I_1t^{M-1}+(t^M)\subset\sO_{X\times\A^1_t}, 
\]
where $\sO_X\supset I_1\supset\cdots\supset I_M$ is a sequence of coherent 
ideal sheaves. 
\item\label{K_dfn2}
Let $\yI$ be a flag ideal and let $s\in\Q_{>0}$. 
A \emph{normal $\Q$-semi test configuration} $(\sB, \sL)/\A^1$ \emph{of} 
$(X, -K_X)$ \emph{obtained by} $\yI$ \emph{and} $s$ is defined by the following datum: 
\begin{itemize}
\item
$\Pi\colon\sB\to X\times\A^1$ is the blowing up along $\yI$ and 
let $E$ be the exceptional divisor, that is, $\sO_{\sB}(-E):=\yI\sO_{\sB}$,
\item
$\sL:=\Pi^*p_1^*(-K_X)-sE$,
\end{itemize}
and we require the following conditions: 
\begin{itemize}
\item
$\sB$ is normal and the morphism $\Pi$ is not an isomorphism, 
\item
$\sL$ is semiample over $\A^1$.
\end{itemize}
\item\label{K_dfn3}
Let $\pi\colon(\sB, \sL)\to\A^1$ be a normal $\Q$-semi test configuration of 
$(X, -K_X)$ obtained by $\yI$ and $s$. For a sufficiently divisible positive integer $k$, 
the multiplicative group $\G_m$ naturally acts on $(\sB, \sO_{\sB}(k\sL))$ and the 
morphism $\pi$ is $\G_m$-equivariant, where the action $\G_m\times\A^1\to\A^1$ 
is in a standard way $(a, t)\mapsto at$. Let
$w(k)$ be the total weight of the induced action on $(\pi_*\sO_{\sB}(k\sL))|_{\{0\}}$ 
and set $N_k:=h^0(X, \sO_X(-kK_X))$. Then $w(k)k'N_{k'}-w(k')kN_k$ is a polynomial 
in variables $k$ and $k'$ for $k$, $k'$ sufficiently divisible positive integers. 
Let $\DF(\sB, \sL)$ be its coefficient in $k^{n+1}{k'}^n$, and is called the 
\emph{Donaldson-Futaki invariant} of $(\sB, \sL)/\A^1$. 
We set $\DF_0:=2((n+1)!)^2\DF(\sB, \sL)/((-K_X)^{\cdot n})$ for simplicity.
\item\label{K_dfn4}
The pair $(X, -K_X)$ is said to be \emph{K-stable} (resp.\ \emph{K-semistable}) 
if $\DF(\sB, \sL)>0$ (resp.\ $\DF(\sB, \sL)\geq 0$) holds 
for any normal $\Q$-semi test configuration $(\sB, \sL)/\A^1$ of $(X, -K_X)$ 
obtained by $\yI$ and $s$. 
\end{enumerate}
\end{definition}

The following is a fundamental result. 

\begin{thm}[{\cite{OS, odk}}]\label{odk_thm}
Let $X$ be a $\Q$-Fano variety of dimension $n$, 
$(\sB, \sL)/\A^1$ be a normal $\Q$-semi test configuration of $(X, -K_X)$ obtained 
by $\yI$ and $s$, and 
$(\bar{\sB}, \bar{\sL})/\pr^1$ be its natural compactification to $\pr^1$, that is, 
$\Pi\colon\bar{\sB}\to X\times\pr^1$ be the blowing up along $\yI$ and 
$\bar{\sL}:=\Pi^*p_1^*(-K_X)-sE$ on $\bar{\sB}$. Then the following holds: 
\begin{enumerate}
\renewcommand{\theenumi}{\arabic{enumi}}
\renewcommand{\labelenumi}{(\theenumi)}
\item\label{odk_thm1}
For a sufficiently divisible positive integer $k$, we have 
\[
w(k)=\chi(\bar{\sB}, \sO_{\bar{\sB}}(k\bar{\sL}))
-\chi(\bar{\sB}, \Pi^*p_1^*\sO_X(-kK_X))+O(k^{n-1}).
\]
In particular, we have 
\[
\lim_{k\to\infty}\frac{w(k)}{kN_k}=\frac{(\bar{\sL}^{\cdot n+1})}{(n+1)((-K_X)^{\cdot n})}.
\]
\item\label{odk_thm2}
We have 
\begin{eqnarray*}
\DF_0 & = & \frac{n}{n+1}(\bar{\sL}^{\cdot n+1})+
(\bar{\sL}^{\cdot n}\cdot K_{\bar{\sB}/\pr^1})\\
 & = & -\frac{1}{n+1}(\bar{\sL}^{\cdot n+1})+
(\bar{\sL}^{\cdot n}\cdot K_{\bar{\sB}/X\times\pr^1}-sE).
\end{eqnarray*}
\item\label{odk_thm3}
We have $(\bar{\sL}^{\cdot n}\cdot E)>0$. 
\item\label{odk_thm4}
If $K_{\bar{\sB}/X\times\pr^1}-sE\geq 0$, then $\DF_0>0$.
\end{enumerate}
\end{thm}

\begin{proof}
\eqref{odk_thm1} and \eqref{odk_thm2} follow from \cite[Proof of Theorem 3.2]{odk}, 
\eqref{odk_thm3} follows from \cite[Lemma 4.5]{OS}, 
and \eqref{odk_thm4} follows from \cite[Proposition 4.4]{OS}.
\end{proof}

\subsection{Multiplier ideal sheaves}\label{multi_section}

We recall the definition and basic properties of multiplier ideal sheaves. 

\begin{definition}\label{M_dfn}
Let $Y$ be a normal $\Q$-Gorenstein variety, 
$\da_1,\dots,\da_l\subset\sO_Y$ be coherent ideal sheaves and 
$c_1,\dots,c_l\in\Q_{\geq 0}$. 
The \emph{multiplier ideal sheaf} $\sI(Y, \da_1^{c_1}\cdots\da_l^{c_l})\subset\sO_Y$ 
of the pair $(Y, \da_1^{c_1}\cdots\da_l^{c_l})$ is defined by the following. 
Take a common log resolution $\mu\colon\hat{Y}\to Y$ of $\da_1,\dots\da_l$, i.e., 
$\hat{Y}$ is smooth, $\da_i\sO_{\hat{Y}}=\sO_{\hat{Y}}(-F_i)$ and 
$\Exc(\mu)$, $\Exc(\mu)+\sum_{1\leq i\leq l}F_i$ are divisors with 
simple normal crossing supports. Then we set 
\[
\sI(Y, \da_1^{c_1}\cdots\da_l^{c_l}):=
\mu_*\sO_{\hat{Y}}(\lceil K_{\hat{Y}/Y}-\sum_{1\leq i\leq l}c_iF_i\rceil), 
\]
where $\lceil K_{\hat{Y}/Y}-\sum_{1\leq i\leq l}c_iF_i\rceil$ is the smallest 
$\Z$-divisor which contains $K_{\hat{Y}/Y}-\sum_{1\leq i\leq l}c_iF_i$. 
\end{definition}

The following proposition can be proved essentially same as the proofs in 
\cite[\S 9]{L}. We omit the proof. 

\begin{proposition}[{see \cite[\S 9]{L}}]\label{M_prop}
We have the following: 
\begin{enumerate}
\renewcommand{\theenumi}{\arabic{enumi}}
\renewcommand{\labelenumi}{(\theenumi)}
\item\label{M_prop1}
$\sI(Y, \da_1^{c_1}\cdots\da_l^{c_l})$ does not depend 
on the choice of $\mu$. 
\item\label{M_prop2}
For an effective Cartier divisor $D$ on $Y$, we have 
\[
\sI(Y, \sO_Y(-D)^1\da_1^{c_1}\cdots\da_l^{c_l})=\sI(Y, \da_1^{c_1}\cdots\da_l^{c_l})\otimes\sO_Y(-D).
\] 
\item\label{M_prop4}
If coherent ideal sheaves $\db_1,\dots,\db_l\subset\sO_Y$ satisfy that 
$\da_i\subset\db_i$ for all $1\leq i\leq l$, then 
\[
\sI(Y, \da_1^{c_1}\cdots\da_l^{c_l})\subset
\sI(Y, \db_1^{c_1}\cdots\db_l^{c_l}).
\] 
\item\label{M_prop5}
Let $Y'$ be another normal $\Q$-Gorenstein variety, 
$\db_1,\dots,\db_{l'}\subset\sO_{Y'}$ be coherent ideal sheaves
and $c'_1,\dots,c'_{l'}\in\Q_{\geq 0}$. Then we have 
\begin{eqnarray*}
& & \sI(Y\times Y', p_1^{-1}\da_1^{c_1}\cdots p_1^{-1}\da_l^{c_l}\cdot
p_2^{-1}\db_1^{c'_1}\cdots p_2^{-1}\db_{l'}^{c'_{l'}})\\
&=& p_1^{-1}\sI(Y, \da_1^{c_1}\cdots\da_l^{c_l})\cdot
p_2^{-1}\sI(Y', \db_1^{c'_1}\cdots\db_{l'}^{c'_{l'}}).
\end{eqnarray*}
\end{enumerate}
\end{proposition}

The following theorem is a singular version of Musta\c{t}\u{a}'s summation 
formula \cite[Corollary 1.4]{M1} due to Shunsuke Takagi. 

\begin{thm}[{\cite[Theorem 3.2]{T}}]\label{T_thm}
Let $Y$ be a normal $\Q$-Gorenstein variety, let 
$\da_0,\da_1,\dots,\da_l\subset\sO_Y$ be coherent ideal sheaves 
and let $c_0$, $c\in\Q_{\geq 0}$. 
Then we have 
\[
\sI\left(Y, \da_0^{c_0}\cdot\Bigr(\sum_{i=1}^l\da_i\Bigl)^c\right)
=
\sum_{\substack{{\scriptstyle c_1+\cdots+c_l=c}\\
{\scriptstyle c_1,\dots,c_l\in\Q_{\geq 0}}}}
\sI\left(Y, \da_0^{c_0}\cdot\prod_{i=1}^l\da_i^{c_i}\right). 
\]
\end{thm}

\section{The projective line case}\label{P1_section}

In this section, we see whether the projective line $\pr^1$ is Berman-Gibbs stable 
or not. 
For any $k\in\Z_{>0}$, we have $N_k=2k+1$ and the morphism associated to 
the complete linear system $|-kK_{\pr^1}|$ is the $(2k)$-th Veronese embedding
$\pr^1\to\pr^{2k}$. 
If the multi-homogeneous coordinates of $(\pr^1)^{2k+1}$ are denoted by 
\[
(t_{1,0}:t_{1,1};\cdots;t_{2k+1,0}:t_{2k+1,1}), 
\]
then the divisor $D_k\subset(\pr^1)^{2k+1}$ corresponds to the following section:  
\[
\det
\begin{pmatrix}
t_{1,0}^{2k} & t_{1,0}^{2k-1}t_{1,1}^1 & \cdots & t_{1,0}^1t_{1,1}^{2k-1} & t_{1,1}^{2k}\\
\vdots     & \vdots                  & \cdots & \vdots                  & \vdots\\
\vdots     & \vdots                  & \cdots & \vdots                  & \vdots\\
t_{2k+1,0}^{2k} & t_{2k+1,0}^{2k-1}t_{2k+1,1}^1 & \cdots & t_{2k+1, 0}^1t_{2k+1,1}^{2k-1} & t_{2k+1,1}^{2k}\\
\end{pmatrix}.
\]
The above matrix is so-called the Vandermonde matrix. Thus, around 
$0\in\A^{2k+1}_{u_1,\dots,u_{2k+1}}\subset(\pr^1)^{2k+1}$, the divisor 
$D_k\subset\A^{2k+1}_{u_1,\dots,u_{2k+1}}$ is defined by the polynomial 
$f_k\in\Bbbk[u_1,\dots,u_{2k+1}]$, where 
\[
f_k:=\prod_{1\leq i<j\leq 2k+1}(u_i-u_j).
\]
By Lemma \ref{mustata_lem}, $\lct_0(\A^{2k+1}, (f_k=0))=2/(2k+1)$. 
Thus 
\[
\lct_{\Delta_{\pr^1}}((\pr^1)^N, (1/k)D_k)=2k/(2k+1).
\] 
Hence $\gamma(\pr^1)=1$. As a consequence, the projective line $\pr^1$ is 
Berman-Gibbs semistable but is not Berman-Gibbs stable.

\begin{lemma}[{\cite{M2}}]\label{mustata_lem}
For $g\geq 2$, we have 
\[
\lct_0\left(\A^g_{u_1,\dots,u_g}, \left(\prod_{1\leq i<j\leq g}\left(u_i-u_j\right)=0\right)\right)=2/g.
\] 
\end{lemma}

\begin{proof}
Set $D:=(\prod_{1\leq i<j\leq g}(u_i-u_j)=0)\subset\A^g$. 
Let $\tau\colon V\to\A^g$ be the blowing up along the line 
$(u_1=\cdots=u_g)$ and let $F$ be its exceptional divisor. For $c\in\Q_{>0}$, 
the discrepancy 
$a(F, \A^g, cD)$ is equal to $g-2-cg(g-1)/2$. Thus $\lct_0(\A^g, D)\leq 2/g$. 
Hence it is enough to show that 
$\lct(\A^g, D)\geq 2/g$. 

Let $H_{ij}\subset\A^g$ be the hyperplane defined by $u_i-u_j=0$ and 
set $\sA:=\{H_{ij}\}_{1\leq i, j\leq g, i\neq j}$. 
We set 
\[
L(\sA):=\Bigl\{W\subset\A^g\,\,\Big|\,\, {}^\exists\sA'\subset\sA;
W=\bigcap_{H\in\sA'}H\Bigr\}.
\]
For $W\in L(\sA)$, set 
$s(W):=\#\{H\in\sA\,|\,W\subset H\}$ and $r(W):=\codim_{\A^g}W$. 
By \cite[Corollary 0.3]{M2}, 
\[
\lct(\A^g, D)=\min_{W\in L(\sA)\setminus\{\A^g\}}\left\{\frac{r(W)}{s(W)}\right\}. 
\]
Pick any $W\in L(\sA)\setminus\{\A^g\}$ and set $r:=r(W)$. 
It is enough to show that $s(W)\leq r(r+1)/2$. If $r=1$, then $s(W)=1$. Thus we can
assume that $r\geq 2$. 
There exist distinct $H_{i_1j_1},\dots,H_{i_rj_r}\in\sA$ such that 
$W=H_{i_1j_1}\cap\dots\cap H_{i_rj_r}$. 

Assume that $i_1,j_1\not\in\{i_2, j_2,\dots,i_r,j_r\}$. 
For any $H_{ij}\in L(\sA)$, if $W\subset H_{ij}$ then $H_{i_1j_1}=H_{ij}$ or 
$H_{i_2j_2}\cap\dots\cap H_{i_rj_r}\subset H_{ij}$. 
Thus $s(W)=1+s(H_{i_2j_2}\cap\dots\cap H_{i_rj_r})\leq 1+r(r-1)/2<r(r+1)/2$ 
by induction on $r$. Hence we can assume that $(i_0:=)i_1=i_2$. 

Assume that $i_0,j_1,j_2\not\in\{i_3, j_3,\dots,i_r,j_r\}$. 
For any $H_{ij}\in L(\sA)$, if $W\subset H_{ij}$ then 
$H_{i_0j_1}\cap H_{i_0j_2}\subset H_{ij}$ or 
$H_{i_3j_3}\cap\dots\cap H_{i_rj_r}\subset H_{ij}$. 
Thus $s(W)=s(H_{i_0j_1}\cap H_{i_0j_2})+s(H_{i_3j_3}\cap\dots\cap H_{i_rj_r})
\leq 2\cdot 3/2+(r-1)(r-2)/2<r(r+1)/2$ 
by induction on $r$. Hence we can assume that $i_3\in\{i_0, j_1, j_2\}$. 
If $i_3=j_1$, then $H_{i_0j_1}\cap H_{j_1j_3}=H_{i_0j_1}\cap H_{i_0j_3}$. 
By replacing $H_{j_1j_3}$ to $H_{i_0j_3}$, we can assume that $(i_0=)i_1=i_2=i_3$. 

We repeat this process. (We note that, for any $1\leq j\leq r-1$, 
$j(j+1)/2+(r-j)(r-j+1)/2<r(r+1)/2$.) We can assume that $(i_0=)i_1=\cdots=i_r$. 
For any $H_{ij}\in L(\sA)$, the condition $W\subset H_{ij}$ is equivalent to the condition 
$\{i,j\}\subset\{i_0,j_1,\dots,j_r\}$. Thus $s(W)=r(r+1)/2$. 
Therefore we have proved that $s(W)\leq r(r+1)/2$. 
\end{proof}

\section{Key propositions}\label{key_section}

In this section, we see the key propositions in order to prove Theorem \ref{mainthm}. 
Throughout the section, let $X$ be a $\Q$-Fano variety of dimension $n$ and let 
$(\sB, \sL)/\A^1$, $\yI$, $s$, and so on are as in Section \ref{K_section}. 

\begin{lemma}\label{W_lem}
Let $k$ be a sufficiently divisible positive integer. 
\begin{enumerate}
\renewcommand{\theenumi}{\arabic{enumi}}
\renewcommand{\labelenumi}{(\theenumi)}
\item\label{W_lem1}
$($cf. \cite[\S 3--4]{RT}$)$ 
Set $I_0:=\sO_X$. We also set 
\[
\tilde{I}_j:=\sum_{\substack{j_1+\cdots+j_{ks}=j\\ 0\leq j_1,\dots,j_{ks}\leq M}}
I_{j_1}\cdots I_{j_{ks}}
\]
for all $0\leq j\leq Mks$. 
Then $\yI^{ks}=\tilde{I}_{Mks}+\tilde{I}_{Mks-1}t+\cdots+\tilde{I}_1t^{Mks-1}+(t^{Mks})$. 
Consider the filtration 
\[
H^0(X, \sO_X(-kK_X))=\sF_0\supset\sF_1\supset\dots\supset\sF_{Mks}\supset 0
\]
defined by $\sF_j:=H^0(X, \sO_X(-kK_X)\cdot\tilde{I}_j)$. 
Set $m:=\sum_{j=1}^{Mks}\dim\sF_j$. Then $m=NMks+w$ holds, 
where $w=w(k)$ and $N=N_k$ are as in Definition \ref{K_dfn} \eqref{K_dfn3}. 
\item\label{W_lem2}
Let $\tilde{I}_{i, j}\subset\sO_{X_i}$ be the copies of $\tilde{I}_j\subset\sO_X$ 
$(X_i:=X)$ for all $1\leq i\leq N$ and set 
\[
J_j:=\sum_{\substack{j_1+\cdots+j_N=j\\ 0\leq j_1,\dots,j_N\leq Mks}}
p_1^{-1}\tilde{I}_{1, j_1}\cdots p_N^{-1}\tilde{I}_{N, j_N}\subset\sO_{X^N}
\]
for all $0\leq j\leq NMks$. 
Then $\sO_{X^N}(-D_k)\subset J_m$ holds. 
\end{enumerate}
\end{lemma}

\begin{proof}
\eqref{W_lem1} By \cite[\S 3--4]{RT}, $(\pi_*\sO_{\sB}(k\sL))|_{\{0\}}$ is equal to
\[
H^0(X\times\A^1_t, \sO(-kK_{X\times\A^1})\cdot\yI^{ks})
/t\cdot H^0(X\times\A^1_t, \sO(-kK_{X\times\A^1})\cdot\yI^{ks})
\]
and is also equal to
\[
\sF_{Mks}\oplus\bigoplus_{j=1}^{Mks}t^j\cdot\Bigl(\sF_{Mks-j}/\sF_{Mks-j+1}\Bigr). 
\]
Thus $w=\sum_{j=1}^{Mks}(-j)(\dim\sF_{Mks-j}-\dim\sF_{Mks-j+1})
=-Mks\dim\sF_0+\sum_{j=1}^{Mks}\dim\sF_j$. This implies that $m=NMks+w$. 

\eqref{W_lem2}
Choose a basis $s_1,\dots,s_N\in H^0(X, \sO_X(-kK_X))$ along the filtration 
$\{\sF_j\}_{0\leq j\leq Mks}$. 
For $1\leq j\leq N$, set 
\[
f(j):=\max\{0\leq i\leq Mks\,|\,s_j\in\sF_i\}.
\] 
Let $s_{i1},\dots,s_{iN}\in H^0(X_i, \sO_{X_i}(-kK_{X_i}))$ be the $i$-th copies of 
$s_1,\dots,s_N$ for all $1\leq i\leq N$. 
Then the divisor $D_k\subset X^N$ corresponds to the section 
\[
\sum_{\sigma\in\dS_N}\sgn\sigma\cdot s_{1\sigma(1)}\cdots s_{N\sigma(N)}
\in H^0(X^N, \sO_{X^N}(-kK_{X^N})), 
\]
where $\dS_N$ is the $N$-th symmetric group. 
Take any $\sigma\in\dS_N$. Since $s_{i, j}\in p_i^{-1}\tilde{I}_{i, f(j)}$, we have 
\[
s_{1\sigma(1)}\cdots s_{N\sigma(N)}\in p_1^{-1}\tilde{I}_{1, f(\sigma(1))}\cdots
p_N^{-1}\tilde{I}_{N, f(\sigma(N))}.
\]
Note that $\sum_{i=1}^Nf(\sigma(i))=\sum_{i=1}^Nf(i)=
\sum_{j=0}^{Mks}j(\dim\sF_j-\dim\sF_{j+1})=m$, where $\sF_{Mks+1}:=0$. 
Thus $\sO_{X^N}(-D_k)\subset J_m$. 
\end{proof}

\begin{proposition}\label{IJ_prop}
Assume that a positive rational number $\gamma\in\Q_{>0}$ satisfies that, 
for a sufficiently divisible positive integer $k$, the pair $(X^N, (\gamma/k)D_k)$ 
is log-canonical around $\Delta_X$. 
Then for any $\varepsilon\in(0, 1)\cap\Q$ and any sufficiently big positive 
integer $P$, the structure sheaf $\sO_{X\times\A^1}$ is contained in the sheaf
\[
\sI\left(X\times\A^1, (t)^{(1-\varepsilon)(1+\gamma w/(kN))+P}\cdot
\yI^{(1-\varepsilon)\gamma s}\right)\otimes\sO_{X\times\A^1}(P\cdot(t=0))
\]
$($that is, the pair $(X\times\A^1, (t)^{(1+\gamma w/(kN))}\cdot\yI^{\gamma s})$ 
is ``sub-log-canonical"$)$, 
where $w=w(k)$ and $N=N_k$ are as in Definition \ref{K_dfn} \eqref{K_dfn3}. 
\end{proposition}

\begin{proof}
We set 
\begin{eqnarray*}
\Theta & := &  \left\{\vec{j}=(j_1,\dots,j_N)\,\,\Big|\,\,\substack{
j_1+\cdots+j_N=m,\\ 0\leq j_1,\dots,j_N\leq Mks}\right\}, \\
A & := & \left\{\vec{\alpha}=(\alpha_{\vec{j}})_{\vec{j}\in\Theta}\,\,\Big|\,\,
\substack{
\sum_{\vec{j}\in\Theta}\alpha_{\vec{j}}=(1-\varepsilon)\gamma/k, \\
{}^\forall\alpha_{\vec{j}}\in\Q_{\geq 0}
}\right\}, \\
B & := & \left\{\vec{\beta}=(\beta_0,\dots,\beta_{Mks})\,\,\Big|\,\,
\substack{
\beta_0,\dots,\beta_{Mks}\in\Q_{\geq 0}, \\
\sum_{j=0}^{Mks}\beta_j=(1-\varepsilon)\gamma/k
}\right\},\\
\Xi & := & \left\{\vec{\xi}=(\xi_0,\dots,\xi_{Mks})\,\,\bigg|\,\, \substack{
\xi_0,\dots,\xi_{Mks}\in\Q_{\geq 0}, \\
\sum_{j=0}^{Mks}\xi_j=(1-\varepsilon)\gamma/k, \\
\sum_{j=0}^{Mks}j\xi_j\geq(1-\varepsilon)\gamma m/(kN)
}\right\}
\end{eqnarray*}
for simplicity. 

\begin{claim}\label{IJ_claim}
We have the equality
\[
\sO_X=\sum_{\vec{\xi}\in\Xi}\sI\left(X, \prod_{i=0}^{Mks}\tilde{I}_i^{\xi_i}\right).
\]
\end{claim}

\begin{proof}[Proof of Claim \ref{IJ_claim}]
By Proposition \ref{M_prop}, Theorem \ref{T_thm} and Lemma \ref{W_lem}, 
around $\Delta_X$, we have 
\begin{eqnarray*}
\sO_{X^N} & = & \sI(X^N, \sO_{X^N}(-D_k)^{(1-\varepsilon)\gamma/k})\\
& \subset & \sI(X^N, J_m^{(1-\varepsilon)\gamma/k})\\
& = & \sI\biggl(X^N, \Bigl(\sum_{\vec{j}\in\Theta}
p_1^{-1}\tilde{I}_{1, j_1}\cdots 
p_N^{-1}\tilde{I}_{N, j_N}\Bigr)^{(1-\varepsilon)\gamma/k}\biggr)\\
& = & \sum_{\vec{\alpha}\in A}
\sI\Bigl(X^N, \prod_{\vec{j}\in\Theta}(p_1^{-1}\tilde{I}_{1, j_1}\cdots 
p_N^{-1}\tilde{I}_{N, j_N})^{\alpha_{\vec{j}}}\Bigr)\\
& = & \sum_{\vec{\alpha}\in A}
p_1^{-1}\sI\bigl(X_1, \prod_{\vec{j}\in\Theta}\tilde{I}_{1, j_1}^{\alpha_{\vec{j}}}\bigr)
\cdots 
p_N^{-1}\sI\bigl(X_N, \prod_{\vec{j}\in\Theta}\tilde{I}_{N, j_N}^{\alpha_{\vec{j}}}\bigr).
\end{eqnarray*}
Restricts to $\Delta_X$, we have 
\[
\sO_X=\sum_{\vec{\alpha}\in A}
\sI\bigl(X, \prod_{\vec{j}\in\Theta}\tilde{I}_{j_1}^{\alpha_{\vec{j}}}\bigr)
\cdots 
\sI\bigl(X, \prod_{\vec{j}\in\Theta}\tilde{I}_{j_N}^{\alpha_{\vec{j}}}\bigr).
\]
Fix an arbitrary $\vec{\alpha}\in A$. 
Since 
\[
\sum_{\vec{j}\in\Theta}\alpha_{\vec{j}}j_1+\cdots+
\sum_{\vec{j}\in\Theta}\alpha_{\vec{j}}j_N=(1-\varepsilon)\gamma m/k, 
\]
we have $\sum_{\vec{j}\in\Theta}\alpha_{\vec{j}}j_q\geq (1-\varepsilon)\gamma m/(kN)$ 
for some $1\leq q\leq N$. We set 
\[
\xi_i:=\sum_{\vec{j}\in\Theta; \,\,j_q=i}\alpha_{\vec{j}}
\]
for $0\leq i\leq Mks$. Then $\vec{\xi}:=(\xi_0,\dots, \xi_{Mks})\in\Xi$ and 
\[
\sI\Bigl(X, \prod_{\vec{j}\in\Theta}\tilde{I}_{j_q}^{\alpha_{\vec{j}}}\Bigr)
=\sI\Bigl(X, \prod_{i=0}^{Mks}\tilde{I}_i^{\xi_i}\Bigr). 
\]
Therefore we have proved Claim \ref{IJ_claim}. 
\end{proof}

By Proposition \ref{M_prop} \eqref{M_prop5} and Claim \ref{IJ_claim}, we have 
\[
\sO_{X\times\A^1}(-P\cdot(t=0))=\sum_{\vec{\xi}\in\Xi}
\sI\Bigl(X\times\A^1, (t)^{1-\varepsilon+P}\cdot\prod_{i=0}^{Mks}\tilde{I}_i^{\xi_i}\Bigr).
\]
For any $\vec{\xi}\in\Xi$, since 
$(1-\varepsilon)(1+\gamma m/(kN))+P-\sum_{i=0}^{Mks}i\xi_i\leq 1-\varepsilon +P$, 
we have 
\begin{eqnarray*}
 &  & 
\sI\Bigl(X\times\A^1, (t)^{1-\varepsilon+P}\cdot\prod_{i=0}^{Mks}\tilde{I}_i^{\xi_i}\Bigr)\\
 & \subset & 
\sI\Bigl(X\times\A^1, (t)^{(1-\varepsilon)
(1+\gamma m/(kN))+P-\sum_{i=0}^{Mks}i\xi_i}\cdot\prod_{i=0}^{Mks}\tilde{I}_i^{\xi_i}\Bigr).
\end{eqnarray*}
On the other hand, by Lemma \ref{W_lem} \eqref{W_lem1} and Theorem \ref{T_thm}, 
we have 
\begin{eqnarray*}
& & \sI\left(X\times\A^1, (t)^{(1-\varepsilon)(1+\gamma w/(kN))+P}
\cdot\yI^{(1-\varepsilon)\gamma s}\right)\\
 & = & \sI\biggl(X\times\A^1, (t)^{(1-\varepsilon)(1-\gamma(Ms-m/(kN)))+P}\cdot
\Bigl(\sum_{i=0}^{Mks}(t)^{Mks-i}\tilde{I}_i\Bigr)^{(1-\varepsilon)\gamma/k}\biggr)\\
 & = & \sum_{\vec{\beta}\in B}
\sI\biggl(X\times\A^1, (t)^{(1-\varepsilon)(1-\gamma(Ms-m/(kN)))+P}\cdot
\prod_{i=0}^{Mks}\Bigl((t)^{Mks-i}\tilde{I}_i\Bigr)^{\beta_i}\biggr)\\
 & = & \sum_{\vec{\beta}\in B}
\sI\Bigl(X\times\A^1, (t)^{(1-\varepsilon)(1+\gamma m/(kN))+P-\sum_{i=0}^{Mks}
i\beta_i}\cdot
\prod_{i=0}^{Mks}\tilde{I}_i^{\beta_i}\Bigr).
\end{eqnarray*}
Since $\Xi\subset B$, we have proved Proposition \ref{IJ_prop}. 
\end{proof}

\section{Proof of Theorem \ref{mainthm}}\label{proof_section}

In this section, we prove Theorem \ref{mainthm}. 
Let $X$ be a $\Q$-Fano variety of dimension $n$ and set $\gamma:=\gamma(X)$. 
We assume that $\gamma\geq 1$. 
Let $(\sB, \sL)/\A^1$ be a normal $\Q$-semi test configuration of $(X, -K_X)$ 
obtained by $\yI$ and $s$ and let $E$, $\bar{\sB}$, $\bar{\sL}$ and so on are 
as in Section \ref{K_section}. 
Let $\{E_\lambda\}_{\lambda\in\Lambda}$ be the set of $\Pi$-exceptional 
prime divisors. 
We note that $\Lambda\neq\emptyset$ since the morphism $\Pi$ is 
not an isomorphism. We set 
\begin{eqnarray*}
\sum_{\lambda\in\Lambda}a_\lambda E_\lambda & := & K_{\bar{\sB}/X\times\pr^1},\\
\sum_{\lambda\in\Lambda}b_\lambda E_\lambda & := & \Pi^*X_0-\hat{X}_0,\\
\sum_{\lambda\in\Lambda}c_\lambda E_\lambda & := & E
\end{eqnarray*}
as in \cite{OS}, where $X_0$ is the fiber of $p_2\colon X\times\pr^1\to\pr^1$ 
at $0\in\pr^1$ and $\hat{X}_0$ is the strict transform of $X_0$ in $\bar{\sB}$. 
We note that $b_\lambda$, $c_\lambda\in\Z_{>0}$ and $a_\lambda-b_\lambda+1>0$ 
for any $\lambda\in\Lambda$ since the pair $(X\times\pr^1, X_0)$ is 
purely-log-terminal. We set 
\[
d:=\max_{\lambda\in\Lambda}\left\{\frac{\gamma sc_\lambda-
(a_\lambda-b_\lambda+1)}{\gamma b_\lambda}\right\}.
\]
By Theorem \ref{odk_thm} \eqref{odk_thm4}, we can assume that $d>0$.

\begin{claim}\label{BGK_claim}
We have the inequality: 
\[
\frac{-(\bar{\sL}^{\cdot n+1})}{(n+1)((-K_X)^{\cdot n})}\geq d.
\]
\end{claim}

\begin{proof}[Proof of Claim \ref{BGK_claim}]
For any sufficiently small positive rational numbers $\varepsilon$ and $\varepsilon'$, 
by Proposition \ref{IJ_prop}, the coefficient of 
\[
K_{\bar{\sB}/X\times\pr^1}-(1-\varepsilon)(1+(\gamma-\varepsilon')w/(kN))\Pi^*X_0
-(1-\varepsilon)(\gamma-\varepsilon')sE
\]
at $E_\lambda$ is strictly bigger than $-1$ for any $\lambda\in\Lambda$ and 
for any sufficiently divisible positive integer $k$. 
Thus, by Theorem \ref{odk_thm} \eqref{odk_thm1}, 
we have 
\[
-1\leq a_\lambda-\left(1-\gamma\frac{-(\bar{\sL}^{\cdot n+1})}{(n+1)
((-K_X)^{\cdot n})}\right)b_\lambda-\gamma sc_\lambda
\]
for any $\lambda\in\Lambda$. Hence we have proved Claim \ref{BGK_claim}. 
\end{proof}

By Claim \ref{BGK_claim}, we have the inequalities: 
\begin{eqnarray*}
\DF_0 & = & \frac{-(\bar{\sL}^{\cdot n+1})}{n+1}+\Bigl(\bar{\sL}^{\cdot n}\cdot
\sum_{\lambda\in\Lambda}(a_\lambda-sc_\lambda)E_\lambda\Bigr)\\
& \geq & \Bigl(\bar{\sL}^{\cdot n}\cdot d\Pi^*X_0
+\sum_{\lambda\in\Lambda}(a_\lambda-sc_\lambda)E_\lambda\Bigr)\\
& = & d(\bar{\sL}^{\cdot n}\cdot\hat{X}_0)+\Bigl(\bar{\sL}^{\cdot n}\cdot 
\sum_{\lambda\in\Lambda}(db_\lambda+a_\lambda-sc_\lambda)E_\lambda\Bigr)\\
& \geq & \Bigl(\bar{\sL}^{\cdot n}\cdot \sum_{\lambda\in\Lambda}(db_\lambda
+a_\lambda-sc_\lambda)E_\lambda\Bigr).
\end{eqnarray*}
For any $\lambda\in\Lambda$, 
\begin{eqnarray*}
&&db_\lambda+a_\lambda-sc_\lambda\geq\frac{1}{\gamma}
\left(\gamma sc_\lambda-(a_\lambda-b_\lambda+1)\right)+a_\lambda-sc_\lambda\\
 & = &\frac{\gamma-1}{\gamma}(a_\lambda-b_\lambda+1)+b_\lambda-1
\geq\frac{\gamma-1}{\gamma}(a_\lambda-b_\lambda+1)
\end{eqnarray*}
holds. Hence 
\[
\DF_0\geq\frac{\gamma-1}{\gamma}\Bigl(\bar{\sL}^{\cdot n}\cdot
\sum_{\lambda\in\Lambda}
(a_\lambda-b_\lambda+1)E_\lambda\Bigr). 
\]
By Theorem \ref{odk_thm} \eqref{odk_thm3}, 
$(\bar{\sL}^{\cdot n}\cdot\sum_{\lambda\in\Lambda}
(a_\lambda-b_\lambda+1)E_\lambda)>0$ holds. 
Therefore, $\DF_0\geq 0$ holds. Moreover, if $\gamma>1$, then $\DF_0>0$ holds. 

As a consequence, we have proved Theorem \ref{mainthm}.

\begin{remark}\label{Berman_rmk}
Robert Berman pointed out to the author that 
there is an analogy between the argument after Claim \ref{BGK_claim} 
and the argument in \cite[Lemma 3.4]{Berman}. 
In fact, the argument in \cite[Lemma 3.4]{Berman} gives the inequality 
\[
\frac{\DF_0}{((-K_X)^{\cdot n})}\geq\frac{-(\bar{\sL}^{\cdot n+1})}
{(n+1)((-K_X)^{\cdot n})}-d_0, 
\]
where 
\[
d_0:=\max\left\{0, \,\,\max_{\lambda\in\Lambda}\left\{\frac{sc_\lambda-a_\lambda}
{b_\lambda}\right\}\right\}.
\]
\end{remark}

\end{document}